\documentclass[12pt,reqno]{amsart}
\usepackage{amsmath,amssymb,amscd}
\usepackage[all]{xy}
\usepackage[latin1]{inputenc}
\usepackage[T1]{fontenc}
\usepackage{hyperref}
\usepackage{ifthen}
\usepackage{fullpage}
\usepackage{subsupscripts}

\newboolean{workmode}
\setboolean{workmode}{false}

\newboolean{sections}
\setboolean{sections}{true}

\input xy
\xyoption{all}

\def\eoe{\unskip\ \hglue0mm\hfill$\diamond$\smallskip\goodbreak}

\makeatletter
\def\th@plain{%
  \thm@notefont{}
  \itshape 
}
\def\th@definition{%
  \thm@notefont{}
  \normalfont 
}
\makeatother

\newcommand{\CC}{{\mathbb{C}}}  
\newcommand{\NN}{{\mathbb{N}}}  
\newcommand{\RR}{{\mathbb{R}}}  
\renewcommand{\SS}{{\mathbb{S}}} 
\newcommand{\TT}{{\mathbb{T}}}  
\newcommand{\ZZ}{{\mathbb{Z}}}  

\newcommand{\Cone}{{\operatorname{Cone}}}
\newcommand{\Der}{{\operatorname{Der}}}  
\newcommand{\diffsp}{\mathbf{DiffSp}} 
\newcommand{\pr}{{\operatorname{pr}}} 
\newcommand{\Stab}{{\operatorname{Stab}}} 
\newcommand{\vect}{{\operatorname{vect}}}  

\newcommand{\g}{{\mathfrak{g}}} 
\renewcommand{\O}{{\operatorname{O}}} 

\newcommand{\CIN}{{C^\infty}}   
\newcommand{\del}{\partial}  
\newcommand{\eps}{\varepsilon}  
\newcommand{\ftimes}[2]{{\lrsubscripts{\times}{#1}{#2}}} 
\newcommand{\toto}{{~\rightrightarrows~}} 
\newcommand{\hook}{{\lrcorner\,}} 
\newcommand{\red}[1]{{/\!/\!_#1\,}} 
\newcommand{\supth}{{^{\text{th}}}} 

\newcommand{\hypref}[2]{{\hyperref[#1]{#2~\ref{#1}}}}

\newcommand{\ifwork}[1]{\ifthenelse{\boolean{workmode}}{#1}{}}
\newcommand{\comment}[1]{}
\newcommand{\mute}[1]{}
\newcommand{\printname}[1]{}

\ifwork{
\renewcommand{\comment}[1]{{\marginpar{*}\ \scriptsize{#1}\ }}
\renewcommand{\printname}[1]
    {\smash{\makebox[0pt]{\hspace{-1.0in}\raisebox{8pt}{\tiny #1}}}}
}

\newcommand{\labell}[1] {\label{#1} \printname{#1}}

\newcommand{\ifsection}[2]{\ifthenelse{\boolean{sections}}{#1}{#2}}

\theoremstyle{plain}
\ifsection{
    \newtheorem{theorem}{Theorem}[section]
}
{
    \newtheorem{theorem}{Theorem}
}

\newtheorem{proposition}[theorem]{Proposition}

\newtheorem{corollary}[theorem]{Corollary}
\newtheorem{lemma}[theorem]{Lemma}

\theoremstyle{definition}
\newtheorem{definition}[theorem]{Definition}
\newtheorem{example}[theorem]{Example}
\newtheorem{remark}[theorem]{Remark}

\newtheorem*{main}{Main Theorem}
\newtheorem*{corollary*}{Corollary}


\title{Symplectic Quotients and Representability:\\
		the Circle Action Case}
		

\author{Jordan Watts}
\address{Department of Mathematics, University of Colorado Boulder, Boulder, Colorado, USA 80309}
\email{jordan.watts@colorado.edu}

\date{\today}

\begin{document}

\begin{abstract}
Let $\SS^1$ act on a symplectic manifold in a Hamiltonian fashion with momentum map $\Psi$.  Fix a value $a$ of $\Psi$.  There is a question of whether the symplectic quotient at $a$ is diffeomorphic to the orbit space of some proper Lie group action.  We prove under mild assumptions that this only occurs if the symplectic quotient is diffeomorphic to an effective orbifold.  This, in turn, only occurs if $a$ is a regular value, or there is at most one positive weight or at most one negative weight.
\end{abstract}

\maketitle

\section{Introduction}\labell{s:intro}

Let $(M,\omega)$ be a symplectic manifold, and let $G$ be a Lie group acting properly on $M$.  The action is \textbf{Hamiltonian} if it  preserves the symplectic form, and there exists a $G$-equivariant\footnote{Some authors do not require the $G$-equivariant condition on a momentum map.  See, for example, \cite{OR}.} smooth map $\Phi\colon M\to\g^*$, where $\g$ is the Lie algebra of $G$ and $\g^*$ is equipped with the co-adjoint action, satisfying 
\begin{equation}\labell{e:Ham}
\xi_M\hook\omega=-d\langle\Phi,\xi\rangle
\end{equation}
for all $\xi\in\g$.  Here, $\xi_M$ is the vector field on $M$ induced by $\xi$.  We call $\Phi$ a \textbf{momentum map} for the action.  A standard operation used when working with Hamiltonian group actions is symplectic reduction, which we describe now.  Fix a value $a$ in the image of $\Phi$.  The \textbf{symplectic reduced space} $M\red{a}G$, or \textbf{symplectic quotient}, is defined to be the space $\Phi^{-1}(a)/G_a$, where $G_a$ is the stabiliser of the co-adjoint action of $G$ on $\g^*$.

If $a$ is a regular value of $\Phi$ and $G_a$ acts freely on $\Phi^{-1}(a)$, then the symplectic quotient $M\red{a}G$ is a symplectic manifold \cite{MW},\cite{meyer}.  In fact, $G_a$ acts at least locally freely on $\Phi^{-1}(a)$ if and only if $a$ is a regular value of the momentum map.  In this case, the symplectic quotient is a symplectic orbifold.  If $a$ is a critical value of the momentum map, however, then $M\red{a}G$ is not necessarily an orbifold.  Sjamaar--Lerman proved that it is a Whitney stratified space equipped with a Poisson algebra of functions inducing a symplectic structure on each stratum; they refer to this as a \textbf{symplectic stratified space} \cite{SL}.

These stratified symplectic quotients are interesting spaces, with or without the symplectic structure, and have been studied in many contexts.  In \cite{HIP}, Herbig--Iyengar--Pflaum specify conditions on a linear Hamiltonian action of a (compact) torus $T$ which guarantee the existence of a deformation quantisation of the symplectic quotient at $0$.  In certain ``admissible'' cases, which include the interesting $\SS^1$-cases, $\Phi^{-1}(0)$ is a cone of a manifold $L$, and it was conjectured by the authors that for most torus actions, $L/T$ is not a rational homology sphere.  The space $L/T$ being a rational homology sphere is a necessary condition for the symplectic quotient to be an orbifold.  These conjectures were resolved in \cite{FHS}, where Farsi--Herbig--Seaton use methods inspired by algebraic geometry behind complex torus actions and their GIT-quotients to find necessary conditions under which the symplectic quotient of a unitary representation of a compact torus is homeomorphic to an orbifold, which involves checking properties of the weight matrix.  These results were extended in \cite{HSS}, where in particular, Herbig--Schwarz--Seaton prove the following for $\SS^1$-actions (see Theorem 1.5 and the discussion afterward in their paper):

\begin{quote}\labell{quote:HSS}
Let $G=\SS^1$ act effectively, linearly, and symplectically on $\CC^n$ equipped with its standard symplectic form, and assume all weights are non-zero.  The symplectic reduced space at $0$, via the standard quadratic momentum map, is regular-diffeomorphic to a linear symplectic (effective) orbifold if and only if the dimension of the reduced space is less than $4$.
\end{quote}

Here, by diffeomorphism we mean in the sense of (Sikorski) differential spaces \cite{sikorski1}, \cite{sikorski2}, \cite{sniatycki}, which we define in Appendix~\ref{a:diff sp} (Definition~\ref{d:diff space}).  Since differential spaces are closed under taking subsets and quotients (see, for example, \cite[Chapter 2]{watts-phd}), there is a natural differential space structure on a symplectic quotient, denoted $\CIN(M\red{a}G)$ (equivalent to the structure studied by Arms--Cushman--Gotay \cite{ACG}), as well as on an orbifold \cite[Section 3]{watts-orb}.  The adjective ``regular'' in the above statement refers (in the linear case) to the diffeomorphism preserving the subalgebra $\RR[\CC^n\red{0}\SS^1]$ of $\CIN(\CC^n\red{0}\SS^1)$ consisting of the image of invariant polynomials via the natural map $\CIN(\CC^n)^{\SS^1}\to\CIN(\CC^n\red{0}\SS^1)$; see \cite[Section 2]{HSS}.  The main result of this paper provides evidence that ``regular'' can be dropped from the above quote.

In the context of Lie groupoids, it is known that (effective) orbifolds are \textbf{representable} by a Lie group action; that is, given an orbifold $X$ with Lie groupoid representative $\mathcal{G}$, there exists a compact Lie group $K$ and a manifold $N$ such that the Lie groupoid $K\ltimes N$ is Morita equivalent to $\mathcal{G}$ (see, for example, \cite[Theorem 2.19]{MM}).  It is shown in \cite[Theorem A and Section 6]{watts-orb} that Morita equivalent Lie groupoids induce diffeomorphic orbit spaces; in fact this correspondence can be realised as a functor $F$ from Lie groupoids to differential spaces that is essentially injective when restricted to Lie groupoids representing orbifolds.  Here, by essentially injective, we mean that if $F(\mathcal{G}_1)$ is diffeomorhpic to $F(\mathcal{G}_2)$, then $\mathcal{G}_1$ is Morita equivalent to $\mathcal{G}_2$.  Using this language, we can now formulate what we mean by a symplectic quotient being representable.

\begin{definition}\labell{d:repr}
Let $(M,\omega)$ be a symplectic manifold, let $G$ be a Lie group acting properly and in a Hamiltonian fashion, and let $a$ be a value in the image of the momentum map.  Then the symplectic quotient $M\red{a}G$ is \textbf{representable} by a proper Lie group action if there exists a Lie group $K$ acting effectively and properly on a manifold $N$, and a diffeomorphism from $N/K$ to $M\red{a}G$.  (Note that $N/K$ is equal to the orbit space of the Lie groupoid $K\ltimes N$.)
\end{definition}

Note that this definition might not be adequate in many circumstances.  Indeed, the question of representability when applied to ineffective orbifolds, for example, requires more than just a diffeomorphism on the level of orbit spaces; information on isotropy groups must also somehow be recorded.  For example, the trivial action of a non-trivial finite group on a point should not be equivalent to the trivial group acting on the point.  This is why representability questions are typically phrased completely in terms of stacks or Lie groupoids, which contain such isotropy information.  And so, perhaps one may wish to take ``representable'' as defined above as a weak form of representability.  However, this is not important for our purposes.  Thus, accepting that our notion of representability, we continue on and focus on the case $G=\SS^1$.

Consider first the case of an effective linear action of $\SS^1$ on $\CC^n$ (assume $\SS^1$ is acting in the standard way with non-zero weights -- see the beginning of Section~\ref{s:proof}).  This preserves the standard symplectic form, and has a homogeneous quadratic momentum map given by Equation~\ref{e:momentum}.  The level set $\Phi^{-1}(0)$ is a cone over a product of ellipsoids $E_-\times E_+$, and $\SS^1$ acts on $E_-\times E_+$ with finite stabilisers; hence, $(E_-\times E_+)/\SS^1$ is an orbifold.  Generalising to effective Hamiltonian actions of $\SS^1$ on a manifold $M$, using local normal forms, this translates to the following:  If $x\in\Phi^{-1}(0)$ is a fixed point and $[x]$ is its orbit, then there is an open neighbourhood of $[x]\in M\red{0}\SS^1$ diffeomorphic to $\RR^{n-2m}\times F$ where $F$ is (stratified-homeomorphic to) a cone over a link $L_x\cong (E_-\times E_+)/\SS^1$ for some ellipsoids $E_-$ and $E_+$.

\begin{definition}\labell{d:unrepresentable}
Let $(M,\omega)$ be a symplectic manifold, let $\SS^1$ act effectively and in a Hamiltonian fashion on $(M,\omega)$, and let $a$ be a value in the image of the momentum map.  Then the symplectic quotient $M\red{a}\SS^1$ is \textbf{unrepresentable} if there does not exist a Lie group $K$ acting effectively and properly on a manifold $N$, and a diffeomorphism from $N/K$ to $M\red{a}\SS^1$.  It is \textbf{weakly unrepresentable} if there does not exist a Lie group $K$ acting effectively and properly on a manifold $N$ with quotient map $\pi_N\colon N\to N/K$, and a diffeomorphism $\psi$ from $N/K$ to $M\red{a}\SS^1$ such that for every $\SS^1$-fixed point $x\in\Phi^{-1}(0)$, the restricted action of $K$ on $(\psi\circ\pi_N)^{-1}(L_x)$ yields an orbifold.
\end{definition}

The extra condition in the definition of weakly unrepresentable is necessary in order to apply our techniques using classifying spaces of Lie groupoids.  It is unknown whether it is needed; it is possible that in the case of representability, it is automatically satisfied.  Indeed, the issue is if the underlying semi-algebraic variety of an orbit space of an effective Lie group action $G\ltimes M$ is diffeomorphic to that of an orbifold, then is $G\ltimes M$ an orbifold?  In general, the answer is no: consider $\O(n)$ acting on $\RR^n$ by rotations and reflections.  The orbit space is $[0,\infty)$ for each $n$, and only in the case $n=1$ do we have an orbifold.  However, there is evidence that the issue is the presence of codimension-$1$ strata (see \cite{CDGMW}), and this issue disappears in the case of symplectic quotients, as each stratum is symplectic, and hence is even-dimensional (there are no codimension-$1$ strata).  This is all purely conjectural at this point, and so we stick with weak unrepresentability and state the main theorem of the paper.

\begin{main}\labell{t:main}
\emph{Let $\SS^1$ act effectively on a symplectic manifold in a Hamiltonian fashion.  Fix a value $a$ of the momentum map and reduce at that value.  Then the resulting symplectic quotient is diffeomorphic to an orbifold, or it is weakly unrepresentable.  If the symplectic quotient is \emph{not} weakly unrepresentable, then either there are no $\SS^1$-fixed points on the $a$-level set of the momentum map, or there is at most one positive weight or at most one negative weight at each $\SS^1$-fixed point on the level set.}
\end{main}

There are no $\SS^1$-fixed points in the $a$-level set if and only if $a$ is a regular value, and we already mentioned above that the resulting symplectic quotient is in this case an orbifold.  In the case that $a$ is a critical value, note that we do not claim that having at most one positive weight or at most one negative weight is a sufficient condition to obtain an orbifold.


The key technique used in the proof of the Main Theorem is assuming that the symplectic quotient is not weakly unrepresentable, and then checking the homotopy groups of classifying spaces of Lie groupoids whose orbit spaces are certain links in the orbit-type stratifications of the corresponding orbit space and symplectic quotient.  These homotopy groups are Morita invariants (or stacky invariants, if you prefer stacks); see Proposition~\ref{p:morita}.  They prove to be powerful tools in studying orbit spaces.  In the case of Lie groupoids representing effective orbifolds, the fundamental groups of the corresponding classifying spaces turn out to be the orbifold fundamental groups in the sense of Thurston, which in fact are much weaker invariants than general stacky ones; indeed, they can be obtained from the underlying differential space structure. (See \cite[Proposition 3.19, Corollary 4.15, Theorem 5.5, and Theorem 5.10]{watts-orb} for a proof of this fact, and \cite[Definition 1.50, Theorem 2.18]{ALR} for a proof that the two notions of fundamental group match.)

This paper is broken down as follows.  Section~\ref{s:proof} provides the proof of the main theorem.  The proof requires some background on differential spaces, as well as Lie groupoids and their classifying spaces.  A review of the necessary theory of differential spaces is given in Appendix~\ref{a:diff sp}.  One necessary ingredient is the minimality of the orbit-type stratification of a symplectic quotient reduced at $0$.  This is a folk theorem, and so we take the opportunity to give a proof of it using \'Sniatycki's theory of vector fields on subcartesian spaces (Theorem~\ref{t:minimal}).  A review of Lie groupoids and their classifying spaces is given in Appendix~\ref{a:gpds}.  Here, we prove another folk theorem: that the classifying space of an action groupoid $G\ltimes M$ is homotopy equivalent to the corresponding Borel construction $EG\times_G M$.  The proof is essentially a slightly more detailed version of what appears in a preprint of Leida \cite{leida} (or at least the author cannot find a published version; it may also be in Leida's PhD thesis).

Finally, it is worth mentioning for the reader who is familiar with diffeology that a lot of the work in this paper using differential spaces could be done in the category of diffeological spaces instead.  Indeed, there is a functor from Lie groupoids to diffeological spaces that is essentially injective when restricted to Lie groupoids representing orbifolds (see \cite{IKZ} and \cite[Theorem B]{watts-orb}).  However, the author finds it more convenient to work with differential spaces in the context of orbit-type stratifications.

\subsection*{Acknowledgements:}
For the author, the question of the representability of a symplectic quotient originally came up during the question period after a talk by Christopher Seaton, to whom the author is grateful for many discussions along with Carla Farsi concerning orbit spaces, symplectic quotients, and Lie groupoids.  The author also thanks Sarah Yeakel for explaining the details behind classifying spaces to him, and Carla Farsi and Markus Pflaum for their encouragement and interest in this project.

\section{Proof of Main Theorem}\labell{s:proof}

In the proof of the main theorem, we will restrict our attention to linear models.  $\SS^1$-actions in this context take on a very nice form: about a fixed point of the $\SS^1$-action on a manifold, there is an $\SS^1$-equivariant neighbourhood symplectomorphic to $\CC^n$ equipped with the standard symplectic form, in which $\SS^1$ acts by $$e^{i\theta}\cdot(z_1,\dots,z_n)=(e^{\alpha_1i\theta}z_1,\dots,e^{\alpha_ni\theta}z_n).$$  Here, $\{\alpha_1,\dots,\alpha_n\}$ is the multi-set of \textbf{weights} of the action.  There is a homogeneous quadratic momentum map $\Phi$ associated with this action:
\begin{equation}\labell{e:momentum}
\Phi(z_1,\dots,z_n)=\frac{1}{2}\sum_{i=1}^n|z_i|^2\alpha_i.
\end{equation}

We prove a number of lemmas regarding linear $\SS^1$-actions on $\CC^n$ in the following.  First, the zero-set of $\Phi$ is a cone provided that there is at least one negative weight and at least one positive weight.

\begin{lemma}\labell{l:cone}
Let $\SS^1$ act effectively, linearly, and symplectically on $\CC^n$ equipped with its standard symplectic form with weights $\{\alpha_1,\dots,\alpha_n\}$ and homogeneous quadratic momentum map $\Phi$ given by Equation~\eqref{e:momentum}.  Assume that $\alpha_1,\dots,\alpha_m<0$ and $\alpha_{m+1},\dots,\alpha_n>0$ for some $0<m<n$.  Then $\Phi^{-1}(0)$ is diffeomorphic to $\Cone(E_-\times E_+)$, where $E_-$ is the ellipsoid in $\CC^m$ given by $\alpha_1|z_1|^2+\dots+\alpha_m|z_m|^2=-1$, and $E_+$ is the ellipsoid in $\CC^{n-m}$ given by $\alpha_{m+1}|z_{m+1}|^2+\dots+\alpha_n|z_n|^2=1$.
\end{lemma}

\begin{proof}
This can be seen immediately after setting Equation~\eqref{e:momentum} equal to $0$.  The differential structure on $\Cone(E_-\times E_+)$ is given in Definition~\ref{d:triviality}.
\end{proof}

Let $\pi_Z\colon\Phi^{-1}(0)\to\CC^n\red{0}\SS^1$ be the quotient map.  Let $v=\pi_Z(0)\in\CC^n\red{0}\SS^1$, the image via $\pi_Z$ of the apex of the cone $\Cone(E_-\times E_+)$.  We prove in the following lemma that if the (linear) symplectic quotient is representable, then the link at $v$ must be diffeomorphic to some quotient of some sphere by a compact Lie group.

\begin{lemma}\labell{l:links}
Given the set up of Lemma~\ref{l:cone}, assume that there exists a manifold $N$ and a Lie group $G$ acting properly on $N$, such that $(N/G,\CIN(N/G))$ is diffeomorphic to $(\CC^{n}\red{0}\SS^1,\CIN(\CC^{n}\red{0}\SS^1))$.  Then there exists an open neighbourhood $U$ of $v$ such that $U\smallsetminus\{v\}$ is diffeomorphic to both $\RR\times((E_-\times E_+)/\SS^1)$ and to $\RR\times(\SS^k/H)$, where $k\leq\dim N$ and $H$ is a closed subgroup of $G$.  Consequently, we have that $(E_-\times E_+)/\SS^1$ is diffeomorphic to $\SS^k/H$.
\end{lemma}

\begin{proof}
Identify $N/G$ and $\CC^{n}\red{0}\SS^1$.  Fix $x\in N$ such that $v=G\cdot x$.  By the slice theorem, there is an open neighbourhood of $x$ that is $G$-equivariantly diffeomorphic to $G\times_H V$ where $H$ is the stabiliser of $x$ and $V$ is the isotropy representation at $x$; identify these.  Let $k+1=\dim V$.  Then, identifying $V$ as $\RR^{k+1}$ equipped with an $H$-invariant metric, we have that $V/H$ is diffeomorphic to an open neighbourhood $U$ of $v$ (see Example~\ref{x:orbitspace}).  At the same time, $V\smallsetminus\{0\}$ is $\SS^1$-equivariantly diffeomorphic to $\RR\times\SS^k$, on which $H$ acts diagonally (trivially on $\RR$).  Thus, $U\smallsetminus\{v\}$ is diffeomorphic to $\RR\times(\SS^k/H)$.

By Lemma~\ref{l:cone}, $\Phi^{-1}(0)$ is a cone with link $E_-\times E_+$.  Note that the link is $\SS^1$-invariant.  Thus, shrinking $U$ if necessary, we have that $U\smallsetminus\{v\}$ is diffeomorphic to $\RR\times((E_-\times E_+)/\SS^1)$.

By Theorem~\ref{t:ots invt} and Corollary~\ref{c:minimal}, the diffeomorphism between $N/G$ and $M\red{0}\SS^1$ preserves the orbit-type stratifications on the spaces.  It then follows that $(E_-\times E_+)/\SS^1$ is diffeomorphic to $\SS^k/H$.
\end{proof}

Since the link at $v$ has two representations $\SS^k/H$ and $(E_-\times E_+)/\SS^1$ that are diffeomorphic, under the assumption of $\CC^n\red{0}\SS^1$ not being weakly unrepresentable, the corresponding Lie groupoids are Morita equivalent.  Furthermore, the corresponding classifying spaces are homotopy equivalent to each other, and to the corresponding Borel constructions.  This is the content of the following lemma.

\begin{lemma}\labell{l:morita}
Given the set up of Lemma~\ref{l:links}, and assuming that $\CC^n\red{0}\SS^1$ is not weakly unrepresentable, we have that the action groupoids $H\ltimes\SS^k$ and $\SS^1\ltimes (E_-\times E_+)$ are Morita equivalent.  Consequently, the Borel constructions $EH\times_H\SS^k$ and $E\SS^1\times_{\SS^1}(E_-\times E_+)$ are homotopy equivalent.
\end{lemma}

\begin{proof}
Since $(E_-\times E_+)/\SS^1$ is an effective orbifold, and by hypothesis $\SS^k/H$ is also an effective orbifold, the Morita equivalence is immediate from Lemma~\ref{l:links} and \cite[Theorem A]{watts-orb}: there is an essentially injective functor from Lie groupoids representing effective orbifolds to differential spaces sending each Lie groupoid to its orbit space equipped with its quotient differential structure.  The homotopy equivalence is immediate from Propositions~\ref{p:morita} and \ref{p:action gpd}.
\end{proof}

Since the Borel constructions are base spaces of certain principal bundles, we can apply the long exact sequence of homotopy groups to these bundles and compare.  This will narrow down which weights allow for our representable scenario assumed in Lemma~\ref{l:links}.  This is the content of the following proposition.

\begin{proposition}\labell{p:main}
Given the set up of Lemma~\ref{l:morita}, there is at most one negative weight or at most one positive weight; moreover, $H$ is finite.
\end{proposition}

\begin{proof}
We have two fibrations:
\begin{gather*}
\SS^1\longrightarrow E\SS^1\times(E_-\times E_+)\longrightarrow E\SS^1\times_{\SS^1}(E_-\times E_+), \text{ and}\\
H\longrightarrow EH\times\SS^k\longrightarrow EH\times_H\SS^k.
\end{gather*}

By Lemma~\ref{l:morita}, $EH\times_H\SS^k$ and $E\SS^1\times_{\SS^1}(E_-\times E_+)$ are homotopy equivalent.  Denote one of them by $X$.  Recall that homotopy group functors respect products, $EG$ is contractible for any topological group $G$, and $E_-$ and $E_+$ are diffeomorphic to $\SS^{l_1}$ and $\SS^{l_2}$, respectively, where $l_1=2m-1$ and $l_2=2(n-m)-1$.  We thus have the following two long exact sequences:
\begin{gather*}
\dots\to\pi_{p+1}(X)\to\pi_p(\SS^1)\to\pi_p(\SS^{l_1})\times\pi_p(\SS^{l_2})\to\pi_p(X)\to\pi_{p-1}(\SS^1)\to\dots, \text{ and}\\
\dots\to\pi_{p+1}(X)\to\pi_p(H)\to\pi_p(\SS^k)\to\pi_p(X)\to\pi_{p-1}(H)\to\dots.
\end{gather*}

Without loss of generality, assume $l_1\leq l_2$.  By hypothesis, $l_1\geq 1$.

Assume for a contradiction that $l_1>1$.  Then, since $l_1$ and $l_2$ are odd, we have $3\leq l_1\leq l_2$.  Counting dimensions, we have that 
\begin{equation}\labell{e:dim}
5\leq l_1+l_2-1=k-\dim H\leq k.
\end{equation}
Inserting this information into the long exact sequences of homotopy groups above, we immediately obtain:
\begin{equation}\labell{e:pi0H}
1\cong\pi_1(X)\cong\pi_0(H),
\end{equation}
\begin{equation}\labell{e:pi1H}
\ZZ\cong\pi_2(X)\cong\pi_1(H),
\end{equation}
\begin{equation}\labell{e:pipH}
\pi_p(\SS^{l_1})\times\pi_p(\SS^{l_2})\cong\pi_p(X)\cong\pi_{p-1}(H) \text{ for $2<p<k$}.
\end{equation}
Equation~\eqref{e:pi0H} implies that $H$ is connected.  Since $H$ is compact we have that $H\cong(\TT^q\times K)/\Gamma$ where $\TT^q$ is the $q$-torus; $K$ is trivial or a compact, connected, semi-simple group; and $\Gamma$ is a finite group.  It follows from the long exact sequence of homotopy groups for the fibration $\Gamma\to\TT^q\times K\to H$ and Equation~\eqref{e:pi1H} that $q=1$, $\pi_1(K)\cong 1$, $\Gamma$ is a cyclic group, and
\begin{equation}\labell{e:pipK}
\pi_p(K)\cong\pi_p(H) \text{ for $p\geq 2$.}
\end{equation}
From Hurewicz' Theorem and general properties of the homology of compact $1$-connected semi-simple Lie groups (namely, $H^2(K;\RR)=0$ and $H^3(K;\RR)\neq 0$), it follows from Equations~\eqref{e:pipH} and \eqref{e:pipK} that
\begin{equation}\labell{e:pi2H}
\pi_3(\SS^{l_1})\times\pi_3(\SS^{l_2})\text{ is finite, and}
\end{equation}
\begin{equation}\labell{e:pi3K}
K\cong 1 \text{ or } \dim K \geq 3 \text{ and } \pi_4(\SS^{l_1})\times\pi_4(\SS^{l_2}) \text{ is infinite.}
\end{equation}
Since $l_1\geq 3$, Equation~\eqref{e:pi2H} in fact forces $l_1>3$.  Since $l_1$ and $l_2$ are odd, we have $l_2\geq l_1\geq 5$. But then, Equation~\eqref{e:pi3K} forces $K\cong 1$.  Thus, $H\cong\SS^1$.

By Equation~\eqref{e:dim}, we have $k\geq 10$.  Assume that $M\in\NN$ such that $10\leq M\leq k$.  Then, by Equation~\eqref{e:pipH}, for all $p=3,\dots,M-1$, we have $\pi_p(\SS^{l_1})\times\pi_p(\SS^{l_2})\cong 1$.  This implies $M\leq l_1\leq l_2$.  It follows from the right-hand-side of Equation~\eqref{e:dim} that $2M\leq k$.  It follows that $k=\infty$, which is absurd.  We conclude that $l_1=1$.

Next, assume that $l_2>1$ (and hence $l_2\geq 3$).  Equation~\eqref{e:dim} reduces to 
\begin{equation}\labell{e:dim2}
3\leq l_2=k-\dim H.
\end{equation}
Looking at the long exact sequences of homotopy groups as above, we obtain the following information:
\begin{equation}\labell{e:exact seq}
1\to\pi_2(X)\to\ZZ\overset{f}{\longrightarrow}\ZZ\to\pi_1(X)\to 1 \text{ is exact,}
\end{equation}
\begin{equation}\labell{e:pipl2}
\pi_p(\SS^{l_2})\cong\pi_p(X) \text{ for $p\geq 3$},
\end{equation}
\begin{equation}\labell{e:pi1X}
\pi_1(X)\cong\ZZ_q \text{ for some $q$, or is trivial,}
\end{equation}
\begin{equation}\labell{e:pi2X}
\pi_2(X)\cong\pi_1(H).
\end{equation}
It follows from Equation~\eqref{e:exact seq} that $\pi_2(X)\cong\ZZ$ or is trivial.  In the former case, $f$ must be the zero homomorphism, in which case $\pi_1(X)\cong\ZZ$, contradicting Equation~\eqref{e:pi1X}.  Thus, $\pi_2(X)\cong 1$, and Equation~\eqref{e:pi2X} implies that the identity component $H_0$ of $H$ is simply-connected. Thus $H$, being compact, cannot have dimension $1$ or $2$.

Assume that $\dim H\geq 3$.  It follows from the long exact sequence of homotopy groups induced by the fibration $\Gamma\to\TT^r\times K\to H_0\cong(\TT^r\times K)/\Gamma$, where $\Gamma$ is finite and $K$ is compact, connected, and semi-simple, that $r=0$ and $\Gamma\cong 1$, and so $H_0\cong K$.  That is, $H_0$ is compact, $1$-connected, and semi-simple.  From Equation~\eqref{e:dim2} we know that $k\geq 6$, and thus by Equation~\eqref{e:pipl2} and the second long exact sequence of homotopy groups above:
\begin{equation}\labell{e:pipl2-2}
\pi_p(\SS^{l_2})\cong\pi_{p-1}(H) \text{ for $p=3,4,5$}.
\end{equation}
Since $H_0$ is compact, $1$-connected and semi-simple, from Equation~\eqref{e:pipl2-2} we have that $\pi_2(H)\cong\pi_3(\SS^{l_2})$ is finite.  Thus, $l_2>3$, and hence $l_2\geq 5$ and $\pi_2(H)\cong 1$.  At the same time, since the cohomology $H^3(H_0;\RR)$ is infinite and $H_0$ is $2$-connected, Hurewicz' Theorem implies that $\pi_3(H)\cong\pi_4(\SS^{l_2})$ is infinite.  This contradicts $l_2\geq 5$.  We are left with the case that $\dim H=0$.
\end{proof}

\begin{proof}[Proof of Main Theorem]
Let $(M,\omega)$ be a symplectic manifold admitting a Hamiltonian $\SS^1$-action with momentum map $\Psi$.  Fix a value $a\in\Psi(M)$.  If  $a$ is a regular value, then it is well-known that $M\red{a}\SS^1$ is representable: we have automatically from Equation~\eqref{e:Ham} that $\SS^1$ acts locally freely on $\Psi^{-1}(a)$, a submanifold of $M$, resulting in the orbifold $\Psi^{-1}(a)/\SS^1$.

Assume that $a$ is a critical value of $\Psi$, and that $M\red{a}\SS^1$ is not weakly unrepresentable: it is diffeomorphic to $N/G$ for some Lie group $G$ and manifold $N$ on which $G$ acts effectively and properly satisfying the orbifold condition near certain links as described in the definition of weakly unrepresentable.  Without loss of generality, assume that $a=0$.  Let $z\in\Psi^{-1}(0)$ be an $\SS^1$-fixed point (which must exist).  Let $\dim M=2n$.  By the slice theorem, there is an $\SS^1$-invariant open neighbourhood of $z$ that is $\SS^1$-equivariantly diffeomorphic to an effective, linear, symplectic representation $V$ of $\SS^1$, with weights $\alpha_1,\dots,\alpha_n$, equipped with the standard quadratic momentum map $\Phi$ as in Equation~\eqref{e:momentum}.  Hereafter we focus on the representation $V$, identifying it with $\CC^n$ for convenience.

Assume without loss of generality that $\alpha_1,\dots,\alpha_j\neq 0$ and $\alpha_{j+1},\dots,\alpha_n=0$ ($0<j\leq n$).  Then the zero-set of $\Phi$ is $\SS^1$-equivariantly isomorphic to $\CC^{n-j}$ or $$\Cone(\SS^l\times\SS^m)\times\CC^{n-j}$$ as smooth stratified spaces for some positive odd integers $l$ and $m$.  The former only occurs if the weights $\alpha_1,\dots,\alpha_j$ are all positive or all negative, and the resulting symplectic quotient is diffeomorphic to $\CC^{n-j}$.  Assuming that there is at least one negative weight and at least one positive weight, the symplectic quotient is diffeomorphic to $\CC^j\red{0}\SS^1\times\CC^{n-j}$, where $\SS^1$ acts on $\CC^j$ via the restricted action on $\CC^j\times\{0\}$.

From the proof of Lemma~\ref{l:links}, $\CC^j\red{0}\SS^1$ is diffeomorphic to $\RR^{k+1}/H$ where $k\leq \dim N$ and $H$ is a compact subgroup of $G$.  By Proposition~\ref{p:main} applied to the $\SS^1$-action on $\CC^j$, we conclude that there is exactly one negative weight or exactly one positive weight, and that $\dim H=0$.  Hence $\CC^j\red{0}\SS^1$ is diffeomorphic to an orbifold $\RR^{k+1}/H$.

Since $z$ is an arbitrary fixed point of $\Psi^{-1}(a)$, it follows that $M\red{a}\SS^1$ is locally diffeomorphic to orbit spaces of linear actions of finite groups.  It follows that $M\red{a}\SS^1$ is an orbifold.  Since every (effective) orbifold is representable, this completes the proof.
\end{proof}

\appendix
\section{Differential Spaces and Minimality of Orbit-Type Stratifications}\labell{a:diff sp}

We begin by reviewing the theory of differential spaces.  For more details, see \cite{sniatycki}.

\begin{definition}[Differential Space]\labell{d:diff space}
Let $X$ be a set.  A \textbf{(Sikorski) differential structure}
on $X$ is a family
$\mathcal{F}$ of real-valued functions satisfying the following two conditions:
\begin{enumerate}
\item\labell{i:smooth compatibility} \textbf{(Smooth Compatibility)} For any positive integer $k$, functions
$f_1,...,f_k\in\mathcal{F}$, and $g\in\CIN(\RR^k)$, the composition
$g(f_1,...,f_k)$ is contained in $\mathcal{F}$.

\item\labell{i:locality} \textbf{(Locality)} Equip $X$ with the initial topology induced by $\mathcal{F}$.  Let $f \colon X\to\RR$ be a function such that there exist an open cover $\{U_\alpha\}$ of $X$ and for each $\alpha$, a function $g_\alpha\in\mathcal{F}$ satisfying $f|_{U_\alpha}=g_\alpha|_{U_\alpha}.$  Then $f\in\mathcal{F}$.
\end{enumerate}
A set $X$ equipped with a differential structure $\mathcal{F}$ is called
a \textbf{(Sikorski) differential space} and is denoted by $(X,\mathcal{F})$.
\end{definition}

\begin{definition}[Smooth Map]\labell{d:smooth map}
Let $(X,\mathcal{F}_X)$ and $(Y,\mathcal{F}_Y)$ be two differential
spaces.  A map $F \colon X\to Y$ is \textbf{smooth}
if $F^*\mathcal{F}_Y\subseteq\mathcal{F}_X$.
The map $F$ is called a \textbf{diffeomorphism}
if it is a bijection and both it and its inverse are smooth.
\end{definition}

\begin{remark}
Differential spaces along with smooth maps form a category, which we denote by $\diffsp$.
\end{remark}

\begin{definition}[Differential Subspace]\labell{d:diff subspace}
Let $(X,\mathcal{F})$ be a differential space, and let $Y\subseteq X$ be
any subset.  Then $Y$ comes equipped with a differential structure $\mathcal{F}_Y$ induced by $\mathcal{F}$ as follows:
A function $f\in\mathcal{F}_Y$ if and only if there is a covering $\{U_\alpha\}$ of $Y$ by open sets of $X$ such that for each $\alpha$, there exists $g_\alpha\in\mathcal{F}$ satisfying $$f|_{U_\alpha\cap Y}=g_\alpha|_{U_\alpha\cap Y}.$$  We call $(Y,\mathcal{F}_Y)$ a \textbf{differential subspace} of $X$.  The initial topology on $Y$ induced by $\mathcal{F}_Y$ coincides with the subspace topology on $Y$ (see \cite[Lemma 2.28]{watts-phd}).  If $Y$ is a closed differential subspace of $\RR^n$, then $\mathcal{F}_Y$ is the set of restrictions of smooth functions on $\RR^n$ to $Y$ (see \cite[Proposition 2.36]{watts-phd}).
\end{definition}

\begin{definition}[Subcartesian Space]\labell{d:subcart}
A \textbf{subcartesian space} is a paracompact, second-countable, Hausdorff differential space $(X,\CIN(X))$ with an open cover $\{U_\alpha\}$ such that for each $\alpha$, there exist $n_\alpha\in\NN$ and a diffeomorphism $\varphi_\alpha\colon U_\alpha\to\tilde{U}_\alpha\subseteq\RR^{n_\alpha}$ onto a differential subspace $\tilde{U}_\alpha$ of $\RR^{n_\alpha}$.
\end{definition}

\begin{definition}[Quotient Differential Structure]\labell{d:diff quotient}
Let $(X,\mathcal{F})$ be a differential space, let $\sim$ be an
equivalence relation on $X$, and let $\pi:X \to X/\!\!\sim$
be the quotient map.
Then $X/\!\!\sim$ obtains a differential
structure $\mathcal{F}_\sim$, called the \textbf{quotient differential structure},
comprising all functions $f\colon X/\!\sim\,\to\RR$ each of whose pullback by $\pi$ is in $\mathcal{F}$.  In general, the functional topology generated by $\mathcal{F}_\sim$ is coarser than the quotient topology.
\end{definition}

\begin{definition}[Product Differential Structure]\labell{d:diff prod}
Let $(X,\mathcal{F})$ and $(Y,\mathcal{G})$ be two differential spaces.
The \textbf{product differential structure} on $X\times Y$
is given by those functions
that locally (with respect to the product topology) are of the form  $F ( f\circ\pr_X , g\circ\pr_Y )$
for $F \in \CIN(\RR^2)$, $f \in \mathcal{F}$, and $g \in \mathcal{G}$.
Here,
$\pr_X$ and $\pr_Y$ are the projections onto $X$ and $Y$, respectively.
In particular, the projection maps are functionally smooth.
\end{definition}

\begin{example}[Manifold]\labell{x:manifold}
Any manifold $M$ has a natural differential structure given by $\CIN(M)$.
\end{example}

\begin{example}[Orbit Space]\labell{x:orbitspace}
Let $G$ be a Lie group acting on a manifold $M$.  Then the quotient differential structure on the orbit space $M/G$ consists of all functions each of which pulls back to a $G$-invariant smooth function on $M$.  If $G$ acts on $M$ properly, then $M/G$ is in fact a subcartesian space (see \cite[Theorem 4.3.4]{sniatycki} for the connected group case).  Indeed, fix $x\in M$.  By the slice theorem \cite{koszul}, \cite{palais}, there exists $G$-invariant open neighbourhood $U$ of $x$, and a $G$-equivariant diffeomorphism from $U$ to $G\times_H V$, where $H$ is the stabiliser of $x$, $V$ is the normal space to the orbit $G\cdot x$ equipped with its isotropy action of $H$, and $H$ acts diagonally on $G\times V$ with quotient $G\times_H V$.  It follows that $U/G$ is an open neighbourhood of $G\cdot x$ in the orbit space $M/G$, and that it is diffeomorphic to $V/H$ equipped with its corresponding quotient differential structure.

We now apply the classical theorem of Schwarz \cite{schwarz}.  By averaging, we may assume that $H$ acts on $V$ orthogonally.  The Hilbert map $\sigma=(\sigma_1,...,\sigma_k)\colon V\to\RR^k$ descends to a proper topological embedding of $V/H$ as a closed subset of $\RR^k$; here $\sigma_1,...,\sigma_k$ is a minimal generating set of the ring of $H$-invariant polynomials on $V$.  Moreover, $\sigma^*(\CIN(\RR^k))=\CIN(V)^H$, which implies that the quotient differential structure on $V/H$ is equal to the subspace differential structure induced by $\RR^k$.  It now follows that $M/G$ is subcartesian.
\eoe
\end{example}

Next, we review stratified spaces, especially in the context of smooth structures.  For more details, see \cite{pflaum}, \cite{sniatycki}.

\begin{definition}[Smooth Decomposition]\labell{d:decomposition}
Let $X$ be a Hausdorff paracompact second-countable topological space.  A partition $\mathcal{M}$ of $X$ is a \textbf{decomposition} if it is locally finite, each element of the partition (called a \textbf{piece}) is locally closed in $X$, each piece is a smooth manifold with respect to the subspace topology, and the following \textbf{frontier condition} is satisfied:
\begin{quote}
For every two distinct pieces $M_1$ and $M_2$ of $\mathcal{M}$, if $\overline{M_1}\cap M_2\neq\emptyset$, then $M_2\subseteq\overline M_1$, and we denote this relationship by $M_2<M_1$.
\end{quote}
We call $(X,\mathcal{M})$ a \textbf{decomposed space}.  If $X$ is a subcartesian space, then a decomposition $\mathcal{M}$ of $X$ is a \textbf{smooth decomposition} if the smooth structure on each piece is induced by $X$.

If $(X,\mathcal{M})$ is a decomposed space, then for $x\in X$, we define the \textbf{depth} of $x$ as $$\operatorname{dp}(x):=\sup\{n\in\NN\mid\exists M_0,\dots,M_n\text{ such that }x\in M_0<\dots<M_n\}.$$

Fixing $X$, decompositions of $X$ form a partial order: $\mathcal{M}\leq \mathcal{N}$ if for every $N\in\mathcal{N}$ there exists $M\in\mathcal{M}$ such that $N\subseteq M$; in which case we say that $\mathcal{M}$ is \textbf{coarser} than $\mathcal{N}$, or that $\mathcal{N}$ is \textbf{finer} than $\mathcal{M}$, or that $\mathcal{N}$ is a \textbf{refinement} of $\mathcal{M}$.
\end{definition}

\begin{definition}[Smooth Stratification]\labell{d:stratification}
Let $X$ be a Hausdorff paracompact second-countable topological space.  A \textbf{stratification} of $X$ is a map $\mathcal{S}$ from $X$ to the set of set germs at points of $X$, $x\mapsto S_x$, satisfying:
\begin{quote}
For every $x\in X$ there is an open neighbourhood $U$ of $x$ and a decomposition $\mathcal{M}_U$ of $U$ so that $S_y$ is equal to the set germ of $M_y$ for each $y\in U$, where $M_y\in\mathcal{M}_U$ is the piece containing $y$.
\end{quote}
We call $(X,\mathcal{S})$ a \textbf{stratified space}.  If $X$ is a subcartesian space, then a stratification $x\mapsto S_x$ is a \textbf{smooth stratification} if each decomposition $\mathcal{M}_U$ above is smooth, and we call $(X,\mathcal{S})$ a \textbf{smooth stratified space}.
\end{definition}

\begin{remark}\labell{r:stratification}
In the literature, other definitions of stratification appear.  The benefit of the above version, due to Mather, is that the notion of a stratification is local.  This allows us to identify decomposed spaces that otherwise are technically different.  For example, consider two decompositions of $\RR$: the first having the origin as one piece and its complement as another, and a second having the origin as one piece, and the two connected components of the complement as two more pieces.  Either one induces the same stratification, which is what we typically prefer.

Given the above definitions, a stratification $\mathcal{S}\colon x\mapsto S_x$ on $X$ induces a unique decomposition $\mathcal{M}$ on $X$ satisfying:
\begin{quote}
For each open set $U$ of $X$, and each decomposition $\mathcal{N}$ inducing the stratification on $U$, the restriction of $\mathcal{M}$ to $U$ is coarser than $\mathcal{N}$.
\end{quote}
The pieces of $\mathcal{M}$ in this case are called \textbf{strata}, and are determined by their dimension and their depth.  See \cite[Proposition 1.2.7]{pflaum} for more details.  We will identify a stratification with this unique decomposition.

Finally, we mention without proof that the notion of smooth structure in \cite[Section 1.3]{pflaum} via charts on a stratified space is equivalent to the differential structure on a smooth stratified space defined here, in the sense that one induces the other.
\end{remark}

\begin{definition}[Smooth Stratified Map]\labell{d:stratified map}
Let $(X,\mathcal{M})$ and $(Y,\mathcal{N})$ be stratified spaces.  A continuous map $f\colon X\to Y$ is \textbf{stratified} if for every stratum $M\in\mathcal{M}$ there exists $N\in\mathcal{N}$ such that $f(M)\subseteq N$, and $f|_M$ is smooth.  If $(X,\mathcal{M})$ and $(Y,\mathcal{N})$ are smooth stratified spaces, and $f$ is both smooth and stratified, then it is called a \textbf{smooth stratified map}.  Finally, if $f$ is a diffeomorphism and stratified such that $f^{-1}$ is stratified, then it is called a \textbf{stratified diffeomorphism}.
\end{definition}

\begin{definition}[Whitney Conditions]\labell{d:whitney}
Fix a manifold $X$ and let $M$ and $N$ be submanifolds of $X$.  Then $M$ and $N$ satisfy \textbf{Whitney Condition A} at $x\in M$ if:
\begin{quote}
For any sequence $(n_k)$ in $N$ converging to $x$, if the sequence of tangent spaces $T_{n_k}N$ converges to a subspace $V\subseteq T_xX$ in the corresponding Grassmannian, then $T_xM\subseteq V$.
\end{quote}
Now assume that $X=\RR^n$.  Then $M$ and $N$ satisfy \textbf{Whitney Condition B} at $x\in M$ if:
\begin{quote}
For any two sequences $(m_k)$ in $M$ and $(n_k)$ in $N$ satisfying: 
\begin{enumerate}
\item $m_k\neq n_k$ for each $k$,
\item $\lim m_k=\lim n_k=x$, 
\item the sequence of line segments from $m_k$ to $n_k$ converges in the corresponding projective space to a line $\ell$,
\item the sequence of tangent spaces $T_{n_k}N$ converges to $V\subseteq T_xX$ in the corresponding Grassmannian;
\end{enumerate}
we have $\ell\subseteq V$.
\end{quote}
We extend Whitney Condition B to subcartesian spaces via their local embeddings in the obvious way, and this is independent of the local embeddings chosen.  We call a stratified subcartesian space \textbf{Whitney stratified} if it satisfies Whitney Condition B, which implies that it satisfies Whitney Condition A; see, for example, \cite[Lemma 1.4.5]{pflaum}.
\end{definition}

\begin{definition}[Smooth Local Triviality and Cone Spaces]\labell{d:triviality}
Let $(X,\mathcal{M})$ be a smooth stratified space.  Then $X$ is \textbf{smoothly locally trivial} if for every $M\in\mathcal{M}$ and $x\in M$,
\begin{enumerate}
\item there is an open neighbourhood $U$ of $x$ such that the partition of $U$ into manifolds $M\cap U$ ($M\in\mathcal{M}$) yields a smooth stratification of $U$,
\item there exists a subcartesian space $X'$ with smooth stratification $\mathcal{M}'$ which contains a singleton set $\{y\}\in\mathcal{M}'$,
\item there exists a stratified diffeomorphism $\varphi:U\to (M\cap U)\times X'$ sending $x$ to $(x,y)$.
\end{enumerate}
Note that the strata of $(M\cap U)\times X'$ are the sets $(M\cap U)\times M'$ where $M'\in\mathcal{M}'$.  Often in applications, the differential subspace $X'\smallsetminus\{y\}$ with its induced smooth stratification is stratified-diffeomorphic to a cylinder $(0,1)\times L$, where $L$ is a smooth stratified space called a \textbf{link}.

We define \textbf{cone spaces} inductively by dimension.  A one-point space is a cone space.  Next, if $(X,\mathcal{M})$ is a smoothly locally trivial smooth stratified space such that each $X'$ above is stratified-diffeomorphic to $\Cone(L)$ for some compact cone space  $L$ that is a differential subspace of $\RR^n$ for some $n$, then $X$ is a cone space.  We define $\Cone(L)$ as follows: since $L$ is a differential subspace of $\RR^n$, we can embed it into some open ball of the $n$-sphere, $\iota\colon L\to\SS^n\subset\RR^{n+1}$.  Then $\Cone(L)$ is the differential subspace of $\RR^{n+1}$ given by $$\Cone(L):=\{tx\in\RR^{n+1}\mid x\in\iota(L),\;t\in[0,\infty)\}.$$
\end{definition}

\begin{example}[Orbit-Type Stratification - Part I]\labell{x:ots1}
Let $G$ be a Lie group acting properly on a manifold $M$.  Define for any closed subgroup $H$ of $G$ the subset of \textbf{orbit-type} $(H)$ by $$M_{(H)}:=\{x\in M\mid \exists g\in G \text{ such that }\Stab_G(x)=gHg^{-1}\}.$$  Then the collection of all connected components of all (non-empty) subsets $M_{(H)}$ induce a smooth Whitney stratification of $M$, called the \textbf{orbit-type stratification} (see \cite[Theorem 2.7.4]{DK}).  Moreover, this stratification descends via the quotient map $\pi\colon M\to M/G$ to a smoothly locally trivial smooth stratification on $M/G$, also called the \textbf{orbit-type stratification} in which the strata are induced by the connected components of $(M/G)_{(H)}:=\pi(M_{(H)})$ as $H$ runs over closed subgroups of $G$ such that $M_{(H)}$ is non-empty (see \cite[Theorem 4.3.5]{sniatycki} -- while the proof is for connected groups, it works for general proper actions as well).

As simple explicit examples, let $\ZZ_n$ act on $\RR^2$ by rotation.  Then, with respect to the orbit-type stratifications, $\RR^2/\ZZ_2$ is a cone space stratified-diffeomorphic to the upper half cone given by $x^2+y^2=z^2,\;z\geq 0$ in $\RR^3$, whereas $\RR^2/\ZZ_n$ ($n>2$) is only a smoothly locally trivial smooth stratified space; all of these have the circle as link.
\eoe
\end{example}

\begin{example}[Orbit-Type Stratification - Part II]\labell{x:ots2}
Let $G$ be a Lie group acting properly and in a Hamiltonian fashion on a symplectic manifold $(M,\omega)$.  Then the inclusion map $j\colon M\red{0}G\to M/G$ induces a Whitney stratification on $M\red{0}G$, called the \textbf{orbit-type stratification}, induced by connected components of $(M\red{0}G)_{(H)}:=j^{-1}(\pi(M_{(H)})$ as $H$ runs over closed subgroups of $G$, making $j$ a smooth stratified map.  This is the main result of \cite{SL}.  In fact, this stratification is smoothly locally trivial (see \cite[Theorem 6.3.4]{sniatycki} for the connected group case; the general case follows from \cite[Section 6]{SL}).
\end{example}

The goal of this section is to prove that the orbit-type stratification of $M/G$ in the case of a proper action of $G$ on $M$, as well as that on $M\red{0}G$ in the case of a proper Hamiltonian action, are invariants of the differential structures on the respective spaces; in particular, the stratifications are preserved by diffeomorphisms.  For orbifolds, this is proven in \cite{watts-orb}.  This requires using \'Sniatycki's theory of vector fields on a subcartesian space, and the fact that these stratifications are minimal among all real-analytic Whitney stratifications on each space.  For $M/G$, minimality  was proven by Bierstone \cite{bierstone1}, \cite{bierstone2}.  It seems to be a folk theorem that the same is true on $M\red{0}G$ (see, for example, \cite[Remark 8.3.3]{OR} or \cite[Remark 2.11]{sjamaar}).  And so a secondary goal of this appendix is to give a proof of this latter claim.  We begin by reviewing tangent bundles and vector fields on subcartesian spaces.

\begin{definition}[Tangent Bundles and Global Derivations]\labell{d:tangent}
Let $X$ be a subcartesian space.  Given a point $x\in X$, a \textbf{derivation} of $\CIN(X)$ at $x$ is a linear map $v\colon\CIN(X)\to\RR$ that satisfies Leibniz' rule: for all $f,g\in\CIN(X)$, $$v(fg)=f(x)v(g)+g(x)v(f).$$  The set of all derivations of $\CIN(X)$ at $x$ forms a vector space, called the \textbf{(Zariski) tangent space} of $x$, and is denoted $T_xX$. Define the \textbf{(Zariski) tangent bundle} $TX$ to be the (disjoint) union $$TX:=\bigcup_{x\in X}T_xX.$$  Denote the canonical projection $TX\to X$ by $\tau$.

The \textbf{tangent cone} of $X$ at $x$, denoted $T^C_xX$, is defined to be the equivalence class via $\sim$ of curves $\gamma\colon\RR\to X$ such that $\gamma(0)=x$, where $\gamma_1\sim\gamma_2$ if $\dot{\gamma_1}(0)=\dot{\gamma_2}(0)$.  Note that $$\dot{\gamma}(0)(f):=\frac{d}{dt}\Big|_{t=0}f\circ\gamma(t)$$ for all $f\in\CIN(X)$, and hence we can identify the $T^C_xX$ as a subset of $T_xX$.  The tangent cone is closed under scalar multiplication by non-negative real numbers.

A \textbf{(global) derivation} of $\CIN(X)$ is a linear map $Y\colon\CIN(X)\to\CIN(X)$ that satisfies Leibniz' rule: for any $f,g\in\CIN(X)$, $$Y(fg)=fY(g)+gY(f).$$  Denote the $\CIN(X)$-module of all derivations by $\Der\CIN(X)$.

Fix $Y\in\Der\CIN(X)$ and $x\in X$.  An \textbf{integral curve} $\exp(\cdot Y)(x)$ of $Y$ through $x$ is a smooth map from a connected subset $I^Y_x\subseteq\RR$ containing $0$ to $X$ such that $\exp(0Y)(x)=x$, and for all $f\in\CIN(X)$ and $t\in I^Y_x$ we have $$\frac{d}{dt}(f\circ\exp(tY)(x))=(Yf)(\exp(tY)(x)).$$  An integral curve is \textbf{maximal} if $I_x^Y$ is maximal among the domains of all such curves.  We adopt the convention that the map $c\colon\{0\}\to X\colon 0\mapsto c(0)$ is an integral curve of every global derivation of $\CIN(X)$.
\end{definition}

\begin{remark}\labell{r:zariski}
$TX$ is a subcartesian space with differential structure generated by functions $f\circ\tau$ and $df$ where $f\in\CIN(X)$ and $d$ is the differential $df(v):=v(f)$.  The projection $\tau$ is smooth with respect to this differential structure (see \cite{LSW} or \cite[Proposition 3.3.3]{sniatycki}).

Given $x\in X$, the dimension of $T_xX$ is invariant under diffeomorphism: if $\varphi\colon X\to X'$ is a diffeomorphism of differential spaces, then $X'$ is a subcartesian space, and the dimension of $T_{\varphi(x)}X'$ is equal to that of $T_xX$.  Indeed, it is not hard to show that the \textbf{pushforward} $\varphi_*\colon TX\to TX'$ sending $v\in T_xX$ to $\varphi_*v\in T_{\varphi(x)}X'$ is a linear isomorphism on each tangent space.  (Recall that for any $f\in\CIN(X')$, we have $\varphi_*v(f)=v(f\circ\varphi).$)

Global derivations of $\CIN(X)$ are exactly the smooth sections of $\tau\colon TX\to X$ (see \cite[Proposition 3.3.5]{sniatycki}).  Moreover, for any $Y\in\Der\CIN(X)$ and any $x\in X$, there exists a unique maximal integral curve $\exp(\cdot Y)(x)$ through $x$ (see \cite[Theorem 3.2.1]{sniatycki}).
\end{remark}

\begin{definition}[Vector Fields and their Accessible Sets]\labell{d:vector fields}
Let $X$ be a subcartesian space. Let $D$ be a subset of $\RR\times X$ containing $\{0\}\times X$ such that $D\cap(\RR\times\{x\})$ is connected for each $x\in X$.  A map $\phi\colon D\to X$ is a \textbf{local flow} if $D$ is open, $\phi(0,x)=x$ for each $x\in X$, and $\phi(t,\phi(s,x))=\phi(t+s,x)$ for all $x\in X$ and $s,t\in\RR$ for which both sides are defined.

A \textbf{vector field} on $X$ is a derivation $Y$ of $\CIN(X)$ such that the map $(t,x)\mapsto\exp(tY)(x)$, sending $(t,x)$ to the maximal integral curve of $Y$ through $x$ evaluated at $t$, is a local flow.  Denote the set of all vector fields on $X$ by $\vect(X)$.

Let $\mathcal{M}$ be a smooth stratification of $X$.  Then the pair $(X,\mathcal{M})$ is said to \textbf{admit local extensions of vector fields} if for any stratum $M\in\mathcal{M}$, any vector field $Y_M$ on $M$, and any $x\in M$, there exist an open neighbourhood $U$ of $x$ and a vector field $Y\in\vect(X)$ such that $Y_M|_{U\cap M}=Y|_{U\cap M}.$  The \textbf{accessible set} or \textbf{orbit} of $\vect(X)$ through a point $x$, denoted $O^X_x$, is the set of all points $y\in X$ such that there exist vector fields $Y_1,...,Y_k$ and real numbers $t_1,...,t_k\in\RR$ satisfying $$y=\exp(t_kY_k)\circ...\circ\exp(t_1Y_1)(x).$$
Denote by $\mathcal{O}_X$ the set of all accessible sets $\{O^X_x\mid x\in X\}$.
\end{definition}

\begin{remark}\labell{r:vector fields}
Let $X$ be a subcartesian space.  Let $X'$ be another differential space, and let $F\colon X'\to X$ be a diffeomorphism.  Then $F$ induces a bijection between $\vect(X')$ and $\vect(X)$.  Indeed, $F$ induces an isomorphism between the derivations of $\CIN(X')$ and those of $\CIN(X)$.  Moreover, if $Z\in\vect(X')$, then $F_*Z$ is a vector field on $X$: $$\frac{d}{dt}\Big|_{t=t_0}F\circ\exp(tZ)(x)=F_*Z|_{F(x)}.$$  The reverse direction also holds, and so the result follows.

Assume that $X$ is locally compact.  Then the set of all vector fields $\vect(X)$ is a $\CIN(X)$-module; that is, for any $f\in\CIN(X)$ and any vector field $Y\in\vect(X)$, the derivation $fY$ is a vector field (see \cite[Corollary 4.71]{watts-phd}).

Let $\mathcal{M}$ be a smoothly locally trivial smooth stratification of $X$.  Then $(X,\mathcal{M})$ admits local extensions of vector fields (see \cite[Theorem 4.5]{LS} or \cite[Proposition 4.1.5]{sniatycki}).  It follows then that the set of accessible sets $\mathcal{O}_X$ forms a smooth decomposition of $X$, which induces a stratification for which $\mathcal{M}$ is a refinement.  In particular, if $\mathcal{M}$ is minimal, then $\mathcal{M}$ is induced by $\mathcal{O}_X$. See \cite[Theorem 4.6]{LS} or \cite[Theorem 4.1.6]{sniatycki}.

Finally, $\mathcal{O}_X$ induces a stratification of $X$ if and only if it is locally finite and each $O\in\mathcal{O}_X$ is locally closed in $X$; see \cite[Theorem 4.3]{LS} or \cite[Corollary 4.1.3]{sniatycki}.
\end{remark}

\begin{example}[Orbit-Type Stratification]\labell{x:ots3}
Let $G$ act properly on a manifold $M$.  Then as mentioned in Example~\ref{x:ots1}, the orbit-type stratification on $M/G$ is smoothly locally trivial.  From Remark~\ref{r:vector fields} it follows that the orbit-type stratification is a refinement of the stratification $\mathcal{O}_{M/G}$ by accessible sets of $\vect(M/G)$.  Since the orbit-type stratification is minimal, it follows that the orbit type stratification is induced by the decomposition $\mathcal{O}_{M/G}$.  See \cite[Section 4.4]{sniatycki} for details for connected groups, or \cite{watts-orb} for the orbifold case.
\end{example}

\begin{theorem}[The Orbit-Type Stratification on $M/G$ is a Differential Space Invariant]\labell{t:ots invt}
Let $G$ act properly on a manifold $M$.  Then the orbit-type stratification on $M/G$ is an invariant of the differential structure $\CIN(M/G)$.
\end{theorem}

\begin{proof}
This follows immediately from Remark~\ref{r:vector fields} and Example~\ref{x:ots3}.
\end{proof}

We now focus on the symplectic quotient $M\red{0}G$ for a Hamiltonian $G$-action.

\begin{proposition}\labell{p:minimal}
Let $(X,\mathcal{N})$ be a smoothly locally trivial smooth stratified space.  For each $x\in X$, denote by $N_x\in\mathcal{N}$ the stratum containing $x$.  If for every $x\in X$ we have $T^C_xX=T_xN_x$, then $\mathcal{N}$ is minimal.
\end{proposition}

\begin{proof}
We prove the contrapositive of the claim.  Assume that $\mathcal{N}$ is not minimal; that is, there exists a stratification $\mathcal{M}<\mathcal{N}$.  Then, there is some $M\in\mathcal{M}$ such that $M$ contains more than one stratum of $\mathcal{N}$.  Assume that for every distinct $N_1,N_2\in\mathcal{N}$ such that $N_1,N_2\subset M$, we have that $\overline{N_1}\cap N_2=\emptyset$.  Then every stratum $N\subseteq M$ is open in $M$, and hence each $N\subseteq M$ is made up of connected components of $M$.  But it follows from the definition of a stratification that $N=M$, a contradiction.
Thus, there must exist distinct $N_1,N_2\in\mathcal{N}$ such that $N_1,N_2\subseteq M$ and $N_1\subset\overline{N_2}$.
Thus $N_1$ and $N_2$ are (immersed) submanifolds of $M$, and $N_1$ is in the topological boundary $\overline{N_2}\smallsetminus N_2$.  Since $\mathcal{N}$ is smoothly locally trivial, it follows that $\dim N_1<\dim M$. 
Hence for any point $x\in N_1$, there is an open neighbourhood $U\subseteq M$ of $x$ and a path $\gamma\colon\RR\to U$ such that $\gamma(0)=x$ and $\dot{\gamma}(0)\in T_xM\smallsetminus T_xN_1$.  Thus, the tangent cone of $X$ at $x$ is not equal to $T_xN_1$.  
\end{proof}

\begin{theorem}\labell{t:minimal}
Let $G$ be a Lie group acting properly and in a Hamiltonian fashion on a symplectic manifold $(M,\omega)$.  Then the orbit-type stratification of $M\red{0}G$ is minimal.
\end{theorem}

\begin{proof}
For a contradiction, assume that $M\red{0}G$ is not minimal.  By Proposition~\ref{p:minimal}, there exist a stratum $N$ of the orbit-type stratification of $M\red{0}G$, $x\in N$, and a smooth curve $\gamma\colon\RR\to M\red{0}G$ such that $\gamma(0)=x$ and $\dot{\gamma}(0)\in T_x(M\red{0}G)\smallsetminus T_xN$.  The inclusion map $j\colon M\red{0}G\to M/G$ is stratified, and moreover satisfies 
\begin{equation}\labell{e:minimal}
j((M\red{0}G)_{(H)})=j(M\red{0}G)\cap(M/G)_{(H)}
\end{equation}
for all closed subgroups $H$ of $G$.  Let $N'$ be the stratum of $M/G$ containing $j(N)$.  By Equation~\eqref{e:minimal} the smooth curve $j\circ\gamma$ satisfies $j_*(\dot{\gamma}(0))\in T_{j(x)}(M/G)\smallsetminus T_{j(x)}N'$, it follows from Proposition~\ref{p:minimal} that the orbit-type stratification of $M/G$ is not minimal.  This is a contradiction.
\end{proof}

\begin{remark}\labell{r:minimal}
Note that in general, for $a\neq 0$, the orbit-type stratification of $M\red{a}G$ is not necessarily minimal.  See \cite[Remark 8.3.3]{OR}.
\end{remark}

\begin{corollary}[The Orbit-Type Stratification on $M\red{0}G$ is Induced by Vector Fields]\labell{c:minimal}
Let $G$ be a Lie group acting properly and in a Hamiltonian fashion on a symplectic manifold $(M,\omega)$.  Then the orbit-type stratification of $M\red{0}G$ is induced by $\mathcal{O}_{M\red{0}G}$, and hence is an invariant of $\CIN(M\red{0}G)$.
\end{corollary}

\begin{proof}
This follows immediately from Theorem~\ref{t:minimal} and Remark~\ref{r:vector fields}.
\end{proof}

\section{Lie Groupoids and their Classifying Spaces}\labell{a:gpds}

In this appendix, our aim is to prove that the classifying spaces of two Morita equivalent Lie groupoids are homotopy equivalent, and in the case of an action groupoid $G\ltimes M$, that the classifying space is homotopy equivalent to the Borel construction $EG\times_G M$.  The latter seems to be another folk theorem, although it follows more-or-less from the work of Segal \cite{segal1}, \cite{segal2}.  See also \cite[Section 2]{LU}, \cite[Section 4]{moerdijk}, and \cite[Section 1.4]{ALR}.

\begin{definition}[Lie Groupoid]\labell{d:gpd}
A \textbf{Lie groupoid} $\mathcal{G}=(G_1\toto G_0)$ is a category in which the objects $G_0$ and the arrows $G_1$ are smooth manifolds (generally we do not require $G_1$ to be Hausdorff; however, for our purposes, $G_1$ always will be Hausdorff), all arrows are invertible, composition $m\colon G_2:=G_1\,\ftimes{s}{t}\, G_1\to G_1$ (the \textbf{multiplication map} $(g_1,g_2)\mapsto g_1g_2$) is smooth, and the \textbf{source map} $s\colon G_1\to G_0$ (sending an arrow to its domain), \textbf{target map} $t\colon G_1\to G_0$ (sending an arrow to its codomain), and \textbf{unit map} $u\colon G_0\to G_1$ (sending an object to its identity arrow) along with the \textbf{inversion map} $\operatorname{inv}\colon G_1\to G_1$ (sending an arrow to its inverse) are all smooth.  Moreover, we require that $s$ and $t$ be submersions.
\end{definition}

\begin{example}\labell{x:action gpd}
Let $G$ be a Lie group acting on a manifold $M$.  Then the \textbf{action groupoid} $G\ltimes M$ corresponding to the action is given by: $(G\ltimes M)_1=G\times M$, $(G\ltimes M)_0=M$, $s(g,x)=x$, $t(g,x)=g\cdot x$, and $m((g_1,x),(g_2,g_1\cdot x))=(g_2g_1,x)$. 
\end{example}

\begin{definition}[Smooth Functor]\labell{d:smooth functor}
A \textbf{smooth functor} between Lie groupoids $F\colon\mathcal{G}\to\mathcal{H}$ is a functor that is smooth as a map between spaces of objects $F_0\colon G_0\to H_0$, and smooth as a map between spaces of arrows $F_1\colon G_1\to H_1$.  A smooth functor $F\colon\mathcal{G}\to\mathcal{H}$ is a \textbf{weak equivalence} if the map $$G_1\to(G_0\times G_0)\,\ftimes{(F,F)}{(s,t)}\,H_1\colon g\mapsto(s(g),t(g),F(g))$$ is a diffeomorphism, and the map $$G_0\,\ftimes{F}{t}\,H_1\to H_0\colon(x,h)\mapsto s(h)$$ is a surjective submersion.
\end{definition}

\begin{remark}\labell{r:smooth functor}
It follows from the definition that a weak equivalence $F\colon\mathcal{G}\to\mathcal{H}$ is fully faithful and essentially surjective; in particular, there exists an equivalence of categories from $\mathcal{G}$ to $\mathcal{H}$.
\end{remark}

\begin{definition}[Morita Equivalence]\labell{d:morita}
Two Lie groupoids $\mathcal{G}$ and $\mathcal{H}$ are \textbf{Morita equivalent} if there exists a Lie groupoid $\mathcal{U}$ and weak equivalences $F\colon\mathcal{U}\to\mathcal{G}$ and $F'\colon\mathcal{U}\to\mathcal{H}$.
\end{definition}

\begin{example}
Let $P\to M$ be a principal $G$-bundle.  Then $G$ acts freely on $P$.  The corresponding action groupoid $G\ltimes P$ is Morita equivalent to the trivial groupoid $M\toto M$ (the groupoid whose arrows are only the units).
\end{example}

\begin{definition}[Classifying Space]\labell{d:class sp}
Let $\mathcal{G}$ be a Lie groupoid.  The \textbf{nerve} of $\mathcal{G}$, denoted by $N\mathcal{G}$, is the simplicial manifold $(G^{(n)})$, where $G^{(0)}=G_0$, $G^{(1)}=G_1$, and for $n>1$, $G^{(n)}$ is the $n\supth$ fibred product over $s$ and $t$ of $G_1$.  The face maps are given by $\del^1_0=s$, $\del^1_1=t$, and for each $n>1$,
\begin{align*}
\del^n_0(g_1,\dots,g_n)=&~(g_2,\dots,g_n),\\
\del^n_i(g_1,\dots,g_n)=&~(g_1,\dots,g_{i-1},g_ig_{i+1},g_{i+2},\dots,g_n) && \text{ for all $0<i<n$, and}\\
\del^n_n(g_1,\dots,g_n)=&~(g_1,\dots,g_{n-1}).
\end{align*}
The degeneracy maps are given by $$\eps^n_i(g_1,\dots,g_n)=(g_1,\dots,g_i,u(s(g_i)),g_{i+1},\dots,g_n)$$ for all $0\leq i\leq n$.  The face and degeneracy maps satisfy all of the usual simplicial identities.

Next, let $\Delta^n$ be the standard $n$-simplex.  We have face maps $\hat{\del}^n_i\colon\Delta^{n-1}\to\Delta^n$ and degeneracy maps $\hat{\eps}^n_i\colon\Delta^{n+1}\to\Delta^n$ satisfying
\begin{align*}
\hat{\del}^n_i(t_0,\dots,t_{n-1})=&~(t_0,\dots,t_{i-1},0,t_i,\dots,t_{n-1}), \text{ and}\\
\hat{\eps}^n_i(t_0,\dots,t_{n+1})=&~(t_0,\dots,t_{i_1},t_i+t_{i+1},t_{i+2},\dots,t_{n+1})
\end{align*}
satisfying the usual simplicial identities.

The \textbf{geometric realisation} of $N\mathcal{G}$ is the topological space $B\mathcal{G}$, the \textbf{classifying space} of $\mathcal{G}$, defined as $$B\mathcal{G}=\left(\coprod_{n\in\NN}\Delta^n\times G^{(n)}\right)\Big/\!\sim$$ where $\sim$ is the equivalence relation generated by
\begin{align*}
\left(\hat{\del}^n_i(\underline{t}),\underline{g}\right)&\sim\left(\underline{t},\del^n_i\left(\underline{g}\right)\right),\\
\left(\hat{\eps}^n_i(\underline{t}),\underline{g}\right)&\sim\left(\underline{t},\eps^n_i\left(\underline{g}\right)\right).
\end{align*}

Finally, a smooth functor between two Lie groupoids $F\colon\mathcal{G}\to\mathcal{H}$ induces a map of nerves $NF\colon N\mathcal{G}\to N\mathcal{H}$, which in turn induces a continuous map between the classifying spaces: $BF\colon B\mathcal{G}\to B\mathcal{H}$.
\end{definition}

We now follow the (unpublished) appendix of Leida \cite[Appendix A]{leida}.

\begin{proposition}[$\pi_n(B\mathcal{G})$ is a Morita Invariant]\labell{p:morita}
Let $\mathcal{G}$ and $\mathcal{H}$ be Lie groupoids.  If $\mathcal{G}$ and $\mathcal{H}$ are Morita equivalent, then $B\mathcal{G}$ and $B\mathcal{H}$ are homotopy equivalent.  In particular, their homotopy groups are isomorphic.
\end{proposition}

\begin{proof}
In fact, Segal \cite[Proposition 2.1]{segal1} shows this in the case of any (small) category in which the sets of objects and arrows are topological spaces, and the structure maps are continuous.  In particular, he shows that a natural transformation is sent by the functor $B$ to an homotopy.  For our case, from the definition of Morita equivalence and Remark~\ref{r:smooth functor} it follows that the natural transformations between the functors determining the Morita equivalence yield an homotopy equivalence on the level of classifying spaces.
\end{proof}

\begin{example}[Trivial Groupoid]\labell{x:trivial gpd}
Let $M$ be a manifold.  Consider the trivial groupoid $M\toto M$.  The corresponding nerve is the simplicial manifold in which $M^{(n)}$ is just a copy of $M$, and all face and degeneracy maps are identity maps.  It follows that $B(M\toto M)$ is homeomorphic to $M$ itself.
\end{example}

\begin{example}[Classifying Space of a Group]\labell{x:class sp}
Let $G$ be a Lie group.  Consider the pair groupoid $G\times G\toto G$ with source and target $s(g_1,g_2)=g_1$ and $t(g_1,g_2)=g_2$.  Since for every $g_1,g_2\in G$, there is a unique arrow from $g_1$ to $g_2$, there is a weak equivalence from the trivial groupoid comprising one point and one arrow to $G\times G\toto G$.  It follows that the classifying space $B(G\times G\toto G)$ is contractible.

Now, let $G$ act on the groupoid $G\times G\toto G$ via $g\cdot g_1=g_1g^{-1}$ on objects and $g\cdot(g_1,g_2)=(g_1g^{-1},g_2g^{-1})$ on arrows.  Then the quotient by this action yields another groupoid $G\toto \{*\}$, which is the groupoid corresponding to the group $G$.  The action of $G$ on $G\times G\toto G$ extends to the nerve, and descends to a free action of $G$ on the classifying space $B(G\times G\toto G)$.  Since the classifying space is contractible, it follows that it is a model for $EG$, and the quotient by $G$ yields the classifying space $BG$ of $G$, which is what we expect $B(G\toto\{*\})$ to be.
\end{example}

\begin{proposition}[$B\mathcal{G}$ for an Action Groupoid]\labell{p:action gpd}
Let $G$ be a Lie group acting on a manifold $M$.  Then the classifying space $B\mathcal{G}$ of the action groupoid $\mathcal{G}:=G\ltimes M$ is homotopy equivalent to the Borel construction $EG\times_G M$.
\end{proposition}

\begin{proof}
Consider the groupoid $G\times G\toto G$ defined in Example~\ref{x:class sp}.  Take the product of this groupoid with the trivial groupoid $M\toto M$.  $G$ acts on the resulting product $G\times G\times M\toto G\times M$ by $g\cdot(g_1,x)=(g_1g^{-1},g\cdot x)$ on objects and $$g\cdot(g_1,g_2,x)=(g_1g^{-1},g_2g^{-1},g\cdot x)$$ on arrows.  One can check that the map from the quotient groupoid $$(G\times G\times M)/G\toto (G\times M)/G$$ to $\mathcal{G}$ sending the object $[g_1,x]$ to $g_1\cdot x$ and the arrow $[g_1,g_2,x]$ to $(g_2g_1^{-1},g_1\cdot x)$ is in fact an isomorphism of groupoids.  Hence their classifying spaces are homeomorphic.

Now, the functor $B$ from Lie groupoids to topological spaces respects products \cite[Section 2]{segal1}, and so applying $B$ to $G\times G\times M\toto G\times M$ yields $EG\times M$.  The $G$ action on $G\times G\times M\toto G\times M$ extends to the nerve and descends to a free action of $G$ on $EG\times M$, in which for all $h\in G$, $$h\cdot\left(\left[\underline{t},\underline{g}\right]\!,x\right)=\left(\left[\underline{t},h\cdot\underline{g}\right]\!,h\cdot x\right).$$  One then constructs a homeomorphism from $EG\times_G M$ to the classifying space of the quotient groupoid $(G\times G\times M)/G\toto(G\times M)/G$.  This completes the proof.
\end{proof}


\end{document}